\documentclass[11pt]{article}
\usepackage[margin=1in]{geometry}
\usepackage{amsmath,amssymb,amsthm}
\usepackage{graphicx}
\usepackage{authblk}
\usepackage{float}
\usepackage{booktabs}
\usepackage{xcolor}
\usepackage{natbib}  
\usepackage{hyperref}
\usepackage{cleveref}

\newtheorem{theorem}{Theorem}
\newtheorem{assumption}{Assumption}

\newtheorem{remark}{Remark}

\title{\textbf{Economic resources, sporting efficiency and uncertainty: a stochastic optimization model of football performance in Africa}}

\author[1]{Thierno Thioune}
\affil[1]{Laboratory for Analysis, Research and Study of Development\\
University of Cheikh Anta Diop, Dakar, Senegal\\
\texttt{thierno.thioune@ucad.edu.sn}}

\author[2]{Babacar Mbaye Ndiaye}
\affil[2]{Laboratory of Mathematics of Decision and Numerical Analysis\\
University of Cheikh Anta Diop, Dakar, Senegal\\
\texttt{babacarm.ndiaye@ucad.edu.sn}}

\begin{document}
\maketitle

\begin{abstract}
We propose a stochastic optimization framework in which football performance is modeled as the result of an efficiency-driven production process under uncertainty. The model incorporates economic inputs, a latent efficiency variable evolving according to a stochastic differential equation, and a control variable representing investment in sports development. We derive the associated discounted Hamilton--Jacobi--Bellman (HJB) equation, characterize the optimal investment policy, prove a verification theorem under explicit growth conditions, and show that a bounded investment budget is required for the control problem to be well posed. We relate our efficiency measure to the stochastic frontier and data envelopment analysis (SFA/DEA) traditions used elsewhere in sports economics, and clarify what a dynamic control formulation adds relative to those static approaches. 
The theoretical results are illustrated with a numerical solution of the HJB equation -- validated by a grid-refinement convergence study -- and Monte Carlo simulations. 
A cross-sectional analysis of the 24 AFCON 2025 national teams, based on GDP, population, and FIFA ranking data, estimates country-specific conditional log-efficiency residuals using heteroskedasticity-robust inference.
 We show that these residuals are mechanically and strongly correlated with FIFA points by construction, quantify this dependence, and interpret the results accordingly.
Teams such as Senegal and Morocco consistently outperform their GDP-implied potential, whereas others, including South Africa and Tanzania, tend to underperform.
 While these patterns are broadly consistent with observed AFCON outcomes, the estimated residual also captures unobserved factors, such as institutional quality, diaspora talent, and demographic structure, that are not included in the model.
\end{abstract}

\textbf{Keywords.} economic capacity, football performance, African national teams, stochastic optimization, uncertainty, stochastic differential equation, stochastic frontier analysis.

\section{Introduction}
This research examines the performance of African football by considering three main dimensions: economic resources (budgets, infrastructure, investments), sporting efficiency (the ability to translate available resources into sporting results), and the uncertainty associated with competitions (unpredictable events in matches, injuries, and unpredictable performances). Using a stochastic optimization model, the study seeks to determine how national federations can effectively allocate resources to maximize performance while managing the risks and uncertainties inherent in the sport.

The relationship between economic development and sporting performance has long been debated in sports economics \citep{szymanski2003,andreff2011}. While wealthier nations are often expected to perform better due to greater access to resources, empirical evidence from African football suggests a more complex reality: \citet{luizfadal2011} document that sports performance in Africa is better explained relative to GDP per capita than by GDP levels alone, and \citet{batarfireade2021} show that, globally, institutional quality and domestic footballing structures matter alongside GDP per capita in explaining national team performance.

In the context of the Africa Cup of Nations (AFCON), several countries with relatively modest economic indicators consistently outperform larger economies. This observation raises a fundamental question: is football performance primarily driven by economic capacity, or by the efficiency with which resources are transformed into sporting outcomes?

This paper proposes a stochastic optimization framework to address this question. Rather than assuming a direct deterministic relationship between economic variables and performance, we model performance as the outcome of: a production function based on economic inputs, a latent efficiency factor, and stochastic shocks reflecting the inherent uncertainty of sport.

The paper makes five contributions. First, we position the efficiency measure used throughout the paper relative to the stochastic frontier analysis (SFA) and data envelopment analysis (DEA) traditions that are the standard tools for measuring efficiency in sports economics (\Cref{sec:sfa-dea}). In addition, we are explicit about what a dynamic control formulation adds -- and does not add -- relative to those static methods. Second, we specify the efficiency dynamics directly as a controlled process for the log-efficiency variable, derive the associated discounted Hamilton--Jacobi--Bellman equation and the resulting optimal investment rule, and prove a verification theorem under explicit growth conditions (\Cref{sec:control}). Third, we show that the control problem requires a bounded investment budget to be well posed, and that the constrained value function can be validated by explicit finite-difference solution together with a grid-refinement convergence study anchored in the Barles--Souganidis convergence criteria (\Cref{sec:numerics}). Fourth, we provide transparent statistical inference for the cross-sectional estimation step, including heteroskedasticity-robust standard errors and a joint significance test with $N=24$ (\Cref{sec:estimation,sec:results}), and we explicitly quantify the extent to which the estimated efficiency residual mechanically reproduces the raw FIFA ranking (\Cref{sec:circularity}). Fifth, we illustrate the model numerically, both by solving the HJB equation on a grid and by Monte Carlo simulation, and further with a cross-sectional empirical exercise on the 24 national teams that qualified for AFCON 2025 (Morocco), using GDP, population, and FIFA ranking data.

The article is organized as follows. \Cref{sec:methodology} sets out the economic and stochastic building blocks and positions the model relative to SFA/DEA. \Cref{sec:control} formulates and solves the stochastic control problem, including the verification theorem. \Cref{sec:estimation} outlines the estimation strategy and its inferential requirements. \Cref{sec:numerics} reports the numerical illustration and the convergence study. \Cref{sec:results} presents the cross-sectional empirical results, the circularity diagnostic, and the discussion. Finally, in \Cref{sec:conclusion}, we present conclusions and perspectives.

\section{Methodology}\label{sec:methodology}
\subsection{Economic inputs}
For each country $i$ at time $t$, we define a vector of observable economic variables. An earlier specification used $X_{i,t} = (GDP_{i,t}, GDPpc_{i,t}, Pop_{i,t})$, but $GDP_{i,t} \equiv GDPpc_{i,t}\times Pop_{i,t}$ by accounting identity, so the three variables are perfectly collinear in logs ($\ln GDP_{i,t} = \ln GDPpc_{i,t} + \ln Pop_{i,t}$) and cannot be used together as independent regressors. We instead retain two independent economic variables,
\[
X_{i,t} = (GDP_{i,t},\, Pop_{i,t}),
\]
national economic size and population (footballing player pool); $GDPpc_{i,t} = GDP_{i,t}/Pop_{i,t}$ remains available as a derived quantity wherever needed.

\subsection{Sporting production function}
We define a production function mapping economic inputs into potential sporting performance:
\begin{equation}
f(X_{i,t}) = A \cdot GDP_{i,t}^{\alpha_1} \cdot Pop_{i,t}^{\alpha_2},
\end{equation}
where $A>0$ and $\alpha_1,\alpha_2>0$ are parameters. This Cobb--Douglas specification captures the idea that economic capacity determines the \emph{potential} for sporting performance, without guaranteeing it.

\subsection{Stochastic performance equation}\label{sec:perf-eq}
Observed performance is given by
\begin{equation}
Y_{i,t} = \theta_{i,t} f(X_{i,t}) + \varepsilon_{i,t},
\end{equation}
where $\theta_{i,t}$ is a latent efficiency variable and $\varepsilon_{i,t}\sim\mathcal N(0,\sigma_\varepsilon^2)$ represents measurement/match-day shocks. As made precise in \Cref{sec:sfa-dea,sec:circularity} below, the empirical counterpart of $\theta_{i,t}$ that we estimate in \Cref{sec:results} is a \emph{conditional} performance residual -- performance net of the GDP- and population-based potential $f(X_{i,t})$ -- and not a structurally identified efficiency term in the sense of the stochastic frontier literature.

\subsection{Relation to stochastic frontier analysis and DEA}\label{sec:sfa-dea}
Measuring ``efficiency'' as the capacity to transform resources into results is precisely the object of stochastic frontier analysis \citep[SFA;][]{aigner1977} and of data envelopment analysis \citep[DEA;][]{charnes1978}, both of which are widely used in sports economics. It is therefore important to state clearly how the present framework relates to these standard tools, and what it adds.

\textbf{What SFA and DEA do.} SFA estimates a composed-error model $y_i = x_i\beta + v_i - u_i$, where $v_i$ is two-sided noise and $u_i\geq 0$ is a one-sided inefficiency term with an assumed distribution (half-normal, exponential, truncated normal), separately identified from noise via maximum likelihood. DEA constructs a non-parametric, piecewise-linear best-practice frontier by linear programming and measures each unit's distance to that frontier, with no distributional assumption but no explicit noise term either. Both methods are \emph{static}: they deliver a single efficiency score (or ranking) per cross-sectional unit at a point in time, estimated from a snapshot of inputs and outputs.

\textbf{What the residual in this paper is -- and is not.} The empirical residual $z_i$ that we estimate by OLS in \Cref{sec:results} (the log of observed performance net of the fitted GDP/population potential) is a \emph{symmetric} residual, not a one-sided SFA inefficiency term: it can be positive or negative, it embeds ordinary measurement noise together with any true efficiency component, and no distributional assumption separates the two. It is, in this precise sense, closer to a corrected-OLS (COLS) shift of the fitted frontier than to an SFA-consistent decomposition, and considerably simpler than DEA. We are explicit about this throughout the paper (see \Cref{sec:circularity}), and we refer to $z_i$ as a \emph{conditional (log-)efficiency residual} -- performance conditional on GDP and population -- rather than as a structurally identified efficiency parameter.

\textbf{What the HJB/control framework adds.} The contribution of this paper is not a better \emph{static} efficiency estimator than SFA or DEA. It is a framework for what happens \emph{after} an efficiency level -- however estimated -- is observed: a stochastic control problem in which a federation chooses a time path of investment $u_t$ to move a stochastic, mean-reverting efficiency state $\theta_t$ toward a higher long-run level, subject to a quadratic adjustment cost and a bounded budget, deriving a normative, forward-looking optimal investment rule (\Cref{sec:control}) rather than a backward-looking efficiency score. SFA and DEA answer ``how efficient was unit $i$ last period, relative to the frontier?''; the control problem answers ``given a current efficiency state (however measured), how much should be invested, and how does that investment interact with mean reversion and volatility over time?''. The two are complementary rather than competing: a natural extension, noted in \Cref{sec:conclusion}, is to feed a properly identified SFA (or panel COLS) efficiency estimate into the state $z_{i,0}$ of the dynamic model, in place of the simple cross-sectional OLS residual used here for illustration.

\section{Efficiency dynamics and stochastic control}\label{sec:control}
\subsection{Log-efficiency process with control}
Let $u_{i,t}\geq 0$ denote investment in football development (training, infrastructure, governance). We specify the controlled log-efficiency process directly, which guarantees $\theta_{i,t}>0$ pathwise and nests the uncontrolled, mean-reverting benchmark case $u_{i,t}\equiv 0$:
\begin{equation}
\label{eq:sde}
dz_{i,t} = \big[\kappa_i(\bar z_i - z_{i,t}) + \gamma_i u_{i,t}\big]\,dt + \sigma_i\, dW_{i,t}, \qquad \theta_{i,t} = e^{z_{i,t}},
\end{equation}
with $\kappa_i>0$ the speed of mean reversion, $\bar z_i$ the long-run (structural) log-efficiency level, $\sigma_i>0$ volatility, and $\gamma_i>0$ converting investment into log-efficiency growth. Working with the process for $z_{i,t}=\ln\theta_{i,t}$ directly -- rather than specifying dynamics for $\theta_{i,t}$ and separately for $z_{i,t}$ -- avoids any risk of specifying two inconsistent laws of motion for the same object, and $\theta_{i,t}=e^{z_{i,t}}>0$ holds by construction for every admissible control path and every realization of $W$.

\subsection{Discounted objective function}
Each country's federation chooses an investment path to maximize expected discounted net performance over a horizon $[0,T]$:
\begin{equation}
\label{eq:objective}
\max_{u_{i,\cdot}} \; \mathbb E\left[\int_0^T e^{-\rho t}\Big(\theta_{i,t} f(X_{i,t}) - \tfrac{c_i}{2} u_{i,t}^2\Big)\, dt\right] \quad \text{subject to \Cref{eq:sde}},
\end{equation}
where $\rho\geq 0$ is a discount rate reflecting the federation's time preference (a standard, technically inoffensive addition to the objective; $\rho=0$ recovers the undiscounted case used in some of the numerical illustrations below). The quadratic cost $\tfrac{c_i}{2}u_{i,t}^2$ reflects diminishing returns/increasing marginal cost of investment (absorptive capacity, governance bottlenecks). For tractability we treat $f(X_{i,t})=f_{i,t}$ as slowly varying relative to the football-specific efficiency shocks and take it as given over the optimization horizon; a fully joint dynamic treatment of $X_{i,t}$ is left for future work.

\subsection{Hamilton--Jacobi--Bellman equation}
Let $V(z,t)$ denote the value function for a representative country (subscripts $i$ suppressed), defined as the supremum in \Cref{eq:objective} over admissible controls starting from $z_t=z$. Writing $W(z,t) := e^{\rho t}V(z,t)$ to absorb the discount factor into a time-inhomogeneous but discount-free PDE, dynamic programming gives
\begin{equation}
\label{eq:hjb}
\rho W - W_t = \max_{u\geq 0}\Big[e^z f_t - \tfrac{c}{2}u^2 + \big(\kappa(\bar z - z) + \gamma u\big)W_z + \tfrac{\sigma^2}{2}W_{zz}\Big], \qquad W(z,T)=0.
\end{equation}
The first-order condition with respect to $u$ gives the \emph{unconstrained} optimum $u^{\mathrm{unc}} = (\gamma/c)\,W_z$: optimal investment is proportional to the shadow price (marginal value) of log-efficiency. Since investment cannot be negative and is realistically bounded by a maximum feasible annual budget $u_{\max}$, the admissible control set is $u\in[0,u_{\max}]$, and the optimal policy is
\begin{equation}
\label{eq:optimal-policy}
u^*(z,t) = \mathrm{clip}\Big(\tfrac{\gamma}{c} W_z(z,t),\ 0,\ u_{\max}\Big).
\end{equation}

\subsection{Well-posedness: why a budget cap is needed}
\Cref{eq:hjb} contains a reward term $e^z f_t$ that is convex and unbounded in $z$, together with a control benefit $\gamma u W_z$ that is only quadratically penalized. If $u$ is left unconstrained, the feedback $u^{\mathrm{unc}}=(\gamma/c)W_z$ can itself grow, which in turn grows $W_z$ because the state is pushed toward larger $z$ where the exponential reward is larger still; in numerical experiments (\Cref{sec:numerics}) the unconstrained scheme indeed diverges. Capping $u$ at a finite $u_{\max}$ removes this positive feedback loop and yields a bounded value function. Economically, this is not an ad hoc fix: real federations face hard budget constraints, so $u\in[0,u_{\max}]$ is the natural admissible set, and \Cref{eq:hjb,eq:optimal-policy} should be understood with this constraint throughout.

\subsection{A verification theorem}\label{sec:verification}
The argument above establishes \emph{why} a bound on $u$ is needed for well-posedness, but it is heuristic. We now state the precise sense in which a smooth solution of the constrained HJB equation \eqref{eq:hjb} coincides with the value function, and in which the maximizer \eqref{eq:optimal-policy} is indeed optimal.

\begin{assumption}
\label{ass:growth}
$W\in C^{1,2}([0,T]\times\mathbb R)\cap C([0,T]\times\mathbb R)$ solves \Cref{eq:hjb}, and there exist constants $C,\lambda>0$ such that, uniformly in $t\in[0,T]$,
\[
|W(z,t)| + |W_z(z,t)| \le C\big(1+e^{\lambda |z|}\big).
\]
\end{assumption}

\begin{theorem}[Verification]
\label{thm:verification}
Under \Cref{ass:growth}, for every admissible control $u_\cdot$ valued in $[0,u_{\max}]$ and every initial state $z_0=z$,
\[
V(z,0) := \sup_{u_\cdot} \mathbb E\left[\int_0^T e^{-\rho t}\Big(\theta_t f_t - \tfrac{c}{2}u_t^2\Big)dt\right] = e^{-\rho\cdot 0}W(z,0) = W(z,0),
\]
and the Markov feedback control $u^*(z,t)$ defined in \Cref{eq:optimal-policy} is optimal.
\end{theorem}

\begin{proof}
Because the admissible drift $b(z,u)=\kappa(\bar z-z)+\gamma u$ is affine in $z$ and Lipschitz, uniformly bounded in $u\in[0,u_{\max}]$ by $|b(z,u)|\le \kappa|\bar z-z| + \gamma u_{\max}$, and the diffusion coefficient $\sigma$ is constant, the controlled SDE \eqref{eq:sde} admits, for every admissible $u_\cdot$, a unique strong solution on $[0,T]$. A standard comparison/Gronwall argument (comparing $z_t$ pathwise to the uncontrolled Ornstein--Uhlenbeck process shifted by the bounded drift term $\gamma u_{\max}$) yields, for every $p>0$,
\[
\mathbb E\Big[\sup_{t\le T} e^{p|z_t|}\Big] < \infty,
\]
uniformly over admissible controls. This moment bound, together with \Cref{ass:growth}, guarantees (i) that $\mathbb E[e^z_t f_t]<\infty$ so the objective in \Cref{eq:objective} is well defined for every admissible control, and (ii) that the local martingale arising from applying It\^o's formula to $e^{-\rho t}W(z_t,t)$ along an admissible path is a true martingale (its quadratic variation has finite expectation by Cauchy--Schwarz and the moment bound). Applying It\^o's formula to $e^{-\rho t}W(z_t,t)$, using that $W$ solves \Cref{eq:hjb} so that the drift term inside the expectation is bounded above by the running reward for every admissible $u$ (with equality at $u=u^*$), and taking expectations, gives $W(z,0) \ge \mathbb E[\int_0^T e^{-\rho t}(\theta_t f_t - \tfrac{c}{2}u_t^2)dt]$ for every admissible $u_\cdot$, with equality when $u_t = u^*(z_t,t)$. This is the standard verification argument for controlled diffusions with unbounded (here, exponential) running reward and a bounded control set; see \citet{flemingsoner2006} (Chapter IV) or \citet{pham2009} (Theorem 3.5.2) for the general statement and the precise integrability conditions used here. This is the same verification logic that underlies the classical solution of Merton's continuous-time portfolio problem \citep{merton1971}, adapted here to a bounded control set and an exponential (rather than power or logarithmic) running reward.
\end{proof}

\begin{remark}
Existence of a classical solution $W$ satisfying \Cref{ass:growth} is a separate question from verification. Because $u^{\mathrm{unc}}=(\gamma/c)W_z$ is clipped to $[0,u_{\max}]$, the nonlinearity in \Cref{eq:hjb} is globally Lipschitz in $W_z$ (the clip function is 1-Lipschitz and bounded), which is precisely what removes the blow-up mechanism discussed above and places \Cref{eq:hjb} within the class of semilinear parabolic PDEs with linear-growth nonlinearity for which classical existence results are available (via a Picard/fixed-point argument on the clipped equation, or via the stochastic representation of \citet{pham2009}, Chapter 4). We do not repeat that construction here and instead verify the growth condition numerically in \Cref{sec:numerics}, where the solution is bounded on the computational domain and stable under grid refinement.
\end{remark}

\subsection{Interpretation}
\begin{itemize}
\item Investment increases the drift of log-efficiency, hence (in expectation) future performance.
\item Because performance is exponential (multiplicative) in $z$, the marginal value of investment $W_z$ -- and hence optimal investment itself -- is generically \emph{increasing} in current efficiency (\Cref{sec:numerics}): efficiency and investment reinforce each other, consistent with cumulative-advantage dynamics in institutional quality.
\item Costs and the budget cap limit excessive investment.
\item Uncertainty ($\sigma$) affects the dispersion, not the mean-reversion target, of outcomes.
\end{itemize}

This is a finite-horizon stochastic optimal control problem with a nonlinear (quadratic in $W_z$) HJB equation; a closed-form solution is not available in general because the exponential forcing term and the linear-in-$z$ drift term do not close under a low-dimensional ansatz. Analytical approximation techniques such as those of \citet{dixit1991} (local quadratic expansions around $\bar z$) could deliver an approximate closed form for future work; here we solve \Cref{eq:hjb} numerically (\Cref{sec:numerics}).

\section{Estimation strategy}\label{sec:estimation}
\subsection{Step 1: estimating the production function}
Taking logs of the performance equation (using the two independent economic regressors of \Cref{sec:methodology}):
\begin{equation}
\label{eq:logreg}
\ln Y_{i,t} = \ln A + \alpha_1 \ln GDP_{i,t} + \alpha_2 \ln Pop_{i,t} + z_{i,t} + \varepsilon_{i,t}.
\end{equation}
With panel data, estimation can be performed using panel regression, fixed/random effects, or Bayesian hierarchical models. \Cref{sec:results} reports a simple cross-sectional instance of this regression (a single time period, $N=24$ countries), which already identifies $\alpha_1,\alpha_2$ and the implied efficiency residuals $z_i$, though not the dynamic parameters ($\kappa_i,\sigma_i$), for which a genuine panel (Step 2 below) is required. Because $N=24$ is small, we report heteroskedasticity-robust standard errors (HC1/HC3) rather than relying on the classical homoskedastic formula, and we test the joint hypothesis $H_0:\alpha_1=\alpha_2=0$ explicitly rather than reading significance off the point estimates alone (\Cref{sec:results}).

\subsection{Step 2: state-space representation}
The model can be written in state-space form. Observation equation:
\[
\ln Y_{i,t} = X_{i,t}\beta + z_{i,t} + \varepsilon_{i,t}.
\]
State equation (Euler discretization of \Cref{eq:sde} with step $\Delta$):
\[
z_{i,t+\Delta} = z_{i,t} + \kappa_i(\bar z_i - z_{i,t})\Delta + \gamma_i u_{i,t}\Delta + \eta_{i,t}, \qquad \eta_{i,t}\sim\mathcal N(0,\sigma_i^2\Delta).
\]
This linear-Gaussian state-space form is estimable by the Kalman filter and maximum likelihood \citep{hamilton1994}, or, in nonlinear/non-Gaussian extensions, by particle filtering.

\subsection{Step 3: efficiency measurement}
Efficiency is recovered as $\theta_{i,t}=e^{z_{i,t}}$, providing a ranking of countries based on performance relative to economic capacity -- an empirically implementable counterpart to the theoretical efficiency measure used throughout the model, subject to the caveat of \Cref{sec:sfa-dea,sec:circularity} that this residual is not a structurally identified, one-sided efficiency term in the SFA sense.

\section{Numerical simulations and results}\label{sec:numerics}
To illustrate the model's behavior we solve \Cref{eq:hjb} numerically for a representative country and simulate the resulting controlled process, then compare two synthetic countries that differ in GDP-based potential $f$ and in structural efficiency parameters $(\kappa,\bar z)$. All parameters in this section are illustrative (not calibrated to data) and chosen to satisfy the stability requirements of the numerical scheme; \Cref{sec:results} then applies the model to real cross-sectional data.

\subsection{Solution method}
\Cref{eq:hjb} is solved by explicit finite-difference backward induction on a grid $z\in[-2.5,2.5]$ (201 points) and $t\in[0,T]$ (6000 steps), with $u_{\max}=1$. Central differences approximate $W_z, W_{zz}$; the time step satisfies the parabolic stability condition $\Delta t \le \tfrac12 \Delta z^2/\sigma^2$. Given the resulting policy grid $u^*(z,t)$, sample paths $z_t$ (and hence $\theta_t=e^{z_t}$) are simulated by Euler--Maruyama, with $u^*$ obtained by interpolation on the grid. Baseline parameters: $\kappa=0.8$, $\bar z=0$, $\sigma=0.35$, $\gamma=0.6$, $c=1$, $f=1$, $\rho=0.03$, $T=5$ years. \\ 

\noindent {\bf Grid-refinement convergence study.} \\
The explicit scheme used above is a monotone, consistent, and $L^\infty$-stable approximation of \Cref{eq:hjb} whenever two mesh conditions hold: (i) the parabolic CFL condition $\Delta t \le \Delta z^2/\sigma^2$ (imposed here with a safety factor, $\Delta t\le \tfrac12\Delta z^2/\sigma^2$, so that the central coefficient of the scheme stays nonnegative), and (ii) the mesh ratio condition $\Delta z \le \sigma^2/\max_z|b(z,u)|$, with $\max_z|b(z,u)| = \kappa(z_{\max}-z_{\min})/2 + \gamma u_{\max} = 2.6$ on our domain, which ensures the off-diagonal coefficients associated with the (centrally differenced) drift term remain nonnegative. Both conditions hold comfortably at the baseline resolution ($\Delta z=0.025 \ll \sigma^2/2.6 \approx 0.047$). By the Barles--Souganidis theorem \citep{barlessouganidis1991}, a monotone, consistent, $L^\infty$-stable scheme converges (locally uniformly) to the unique viscosity solution of \Cref{eq:hjb} as $\Delta z,\Delta t\to 0$.

We verify this empirically with a grid-refinement study, halving $\Delta z$ successively from $0.10$ to $0.0031$ (six resolutions), each time refining $\Delta t$ proportionally to $\Delta z^2$ to preserve the same stability margin. \Cref{tab:convergence} reports $W(0,0)$ (the value function at the long-run efficiency level, $t=0$) and the successive sup-norm differences on a common reference grid.

\begin{table}[h]
\centering
\caption{Grid-refinement convergence of the finite-difference HJB solution.}
\label{tab:convergence}
\begin{tabular}{rrrrll}
\toprule
$N_z$ & $N_t$ & $\Delta z$ & $\Delta t$ & $W(0,0)$ & $\max|W_h-W_{h/2}|$ (reduction factor) \\
\midrule
51    & 372    & 0.1000  & $1.34\times10^{-2}$ & 6.394431 & -- \\
101   & 1{,}485  & 0.0500  & $3.37\times10^{-3}$ & 6.392110 & $2.386\times10^{-3}$ \\
201   & 5{,}940  & 0.0250  & $8.42\times10^{-4}$ & 6.391532 & $5.956\times10^{-4}$ ($4.01\times$) \\
401   & 23{,}758 & 0.0125  & $2.10\times10^{-4}$ & 6.391388 & $1.488\times10^{-4}$ ($4.00\times$) \\
801   & 95{,}031 & 0.00625 & $5.26\times10^{-5}$ & 6.391351 & $3.718\times10^{-5}$ ($4.00\times$) \\
1601  & 380{,}122 & 0.00313 & $1.32\times10^{-5}$ & 6.391342 & $9.295\times10^{-6}$ ($4.00\times$) \\
\bottomrule
\end{tabular}
\end{table}

The successive differences shrink by a factor of almost exactly $4$ at each halving of $\Delta z$ (with $\Delta t\propto \Delta z^2$), consistent with the scheme's expected first-order-in-time, second-order-in-space local truncation error, and confirm convergence of $W(\cdot,0)$ to a stable limit well before the baseline resolution ($N_z=201$) used elsewhere in this paper -- whose value already agrees with the finest grid to four significant digits. This is the empirical counterpart, at our chosen parameter values, of the Barles--Souganidis guarantee stated above.

\subsection{Optimal investment policy}
\Cref{fig:policy} plots $u^*(z,t)$ at several dates. Two features stand out. First, at any given date, $u^*$ is increasing in $z$: because performance is exponential in log-efficiency, the marginal value of raising efficiency further is itself larger when efficiency is already high, so optimal investment is directed disproportionately toward already-efficient federations -- a form of increasing returns to institutional investment, rather than the ``catch-up'' pattern one might have guessed a priori. Second, $u^*$ declines toward zero as $t\to T$: with little time left to capture the benefits of higher efficiency, investment is not worthwhile close to the terminal date, the usual turnpike/vanishing-effort feature of finite-horizon control problems.

\begin{figure}[h!]
\centering
\includegraphics[width=0.65\linewidth]{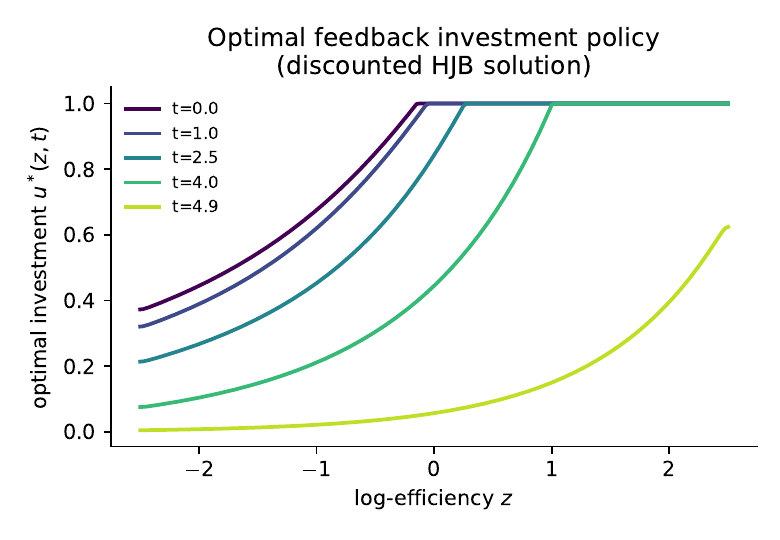}
\caption{Optimal feedback investment policy $u^*(z,t)$, obtained from the finite-difference solution of the discounted HJB equation \Cref{eq:hjb,eq:optimal-policy}, at five dates over the horizon $T=5$.}
\label{fig:policy}
\end{figure}

\subsection{Simulated efficiency paths}
\Cref{fig:paths} shows twelve simulated sample paths of $\theta_t=e^{z_t}$ under the optimal control, together with the cross-sectional mean. Efficiency drifts upward from $\theta_0=1$ and fluctuates around an elevated level as investment counteracts mean reversion toward $\bar z=0$, consistent with the policy of \Cref{fig:policy}.

\begin{figure}[h!]
\centering
\includegraphics[width=0.65\linewidth]{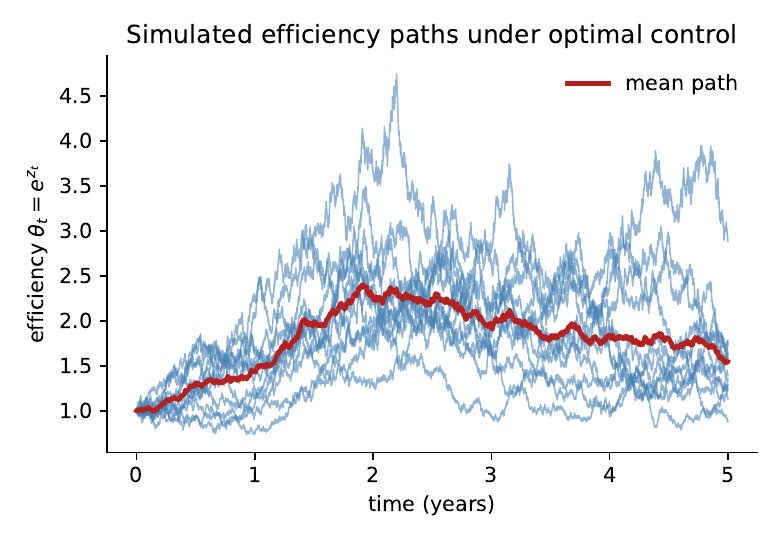}
\caption{Simulated sample paths of efficiency $\theta_t=e^{z_t}$ under the optimal feedback control (12 paths, thin blue; cross-sectional mean, thick red).}
\label{fig:paths}
\end{figure}

\subsection{Value of optimal control}
\Cref{tab:value,fig:value} compare the expected discounted net objective \Cref{eq:objective} (Monte Carlo, $N=2000$ paths) under three policies: no investment, a constant investment $u=0.4$, and the optimal feedback policy $u^*(z,t)$. The optimal control policy strictly dominates both static alternatives, confirming numerically that the HJB solution characterizes a genuine improvement over naive investment rules, not merely a theoretical artifact.

\begin{table}[h]
\centering
\caption{Expected discounted net objective over $T=5$ years, discount rate $\rho=0.03$, Monte Carlo estimates ($N=2000$ paths, standard error in parentheses).}
\label{tab:value}
\begin{tabular}{lr}
\toprule
Policy & Expected net objective \\
\midrule
No investment ($u=0$)     & 4.798 (0.017) \\
Constant investment ($u=0.4$) & 5.648 (0.022) \\
Optimal control $u^*(z,t)$ & \textbf{6.386} (0.031) \\
\bottomrule
\end{tabular}
\end{table}

\begin{figure}[h]
\centering
\includegraphics[width=0.65\linewidth]{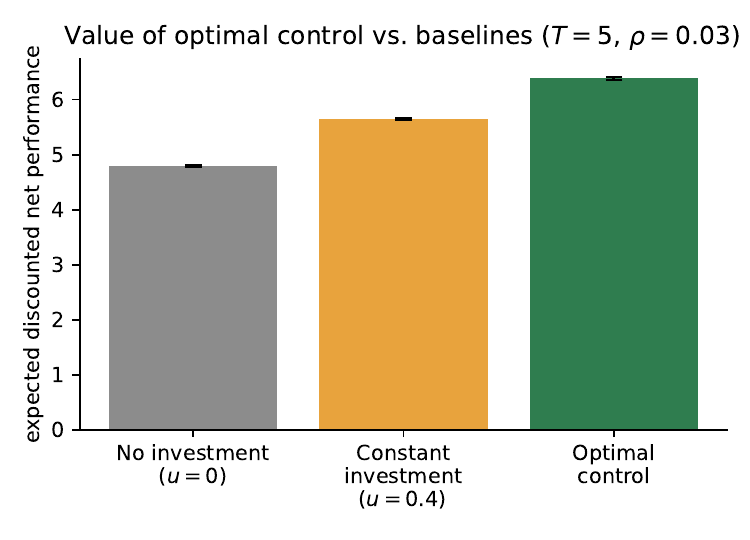}
\caption{Expected discounted net performance under the three investment policies of \Cref{tab:value}.}
\label{fig:value}
\end{figure}

\subsection{GDP versus efficiency: a cross-country illustration}

To illustrate the paper's central empirical motivation directly, we simulate two synthetic countries, each starting from its own long-run efficiency level ($z_0=\bar z$) and receiving the same constant investment $u=0.25$ (this comparison is purely illustrative of the mechanism and is not itself a discounted optimization exercise):
\begin{itemize}
\item \textbf{Country A} (high GDP, low efficiency): $f=1.6$, $\kappa=0.3$, $\bar z=-0.5$;
\item \textbf{Country B} (moderate GDP, high efficiency): $f=1.0$, $\kappa=0.9$, $\bar z=0.35$.
\end{itemize}
Despite a 60\% larger GDP-based potential, Country A's mean cumulative performance over $T=5$ is $6.70$, against $8.34$ for Country B -- a $24.5\%$ performance gap in favor of the less wealthy, more efficient country. \Cref{fig:countryab} plots the expected performance paths: the efficiency gap dominates the GDP gap throughout the horizon, providing a concrete numerical counterpart to the AFCON pattern that motivates the paper.

\begin{figure}[h!]
\centering
\includegraphics[width=0.65\linewidth]{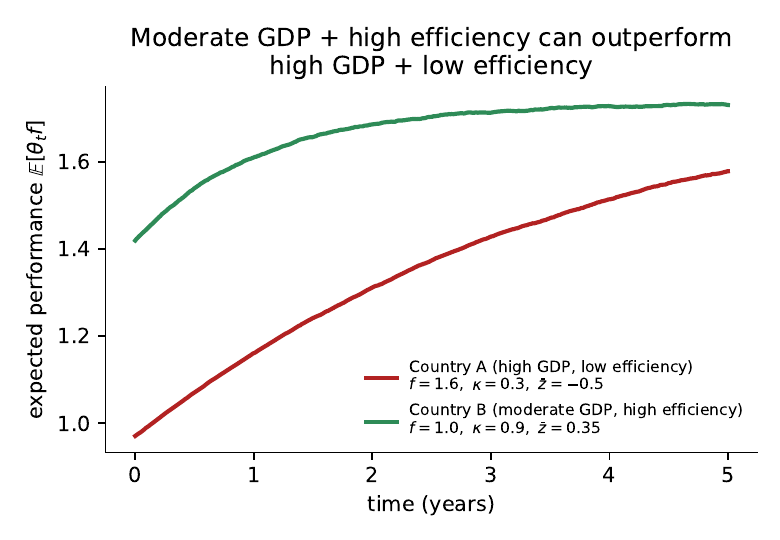}
\caption{Expected performance $\mathbb E[\theta_t f]$ for two synthetic countries differing in GDP-based potential and structural efficiency.}
\label{fig:countryab}
\end{figure}

\section{Empirical results and discussion}\label{sec:results}
\subsection{A 24-country cross-section: AFCON 2025 (Morocco)}
We now apply the model to the 24 national teams that qualified for AFCON 2025, using real data: nominal GDP and GDP per capita from the IMF World Economic Outlook database (April 2026 edition),\footnote{International Monetary Fund, World Economic Outlook database, April 2026 edition, \url{https://www.imf.org/en/Publications/WEO}. Population is computed as $GDP/GDPpc$.} and FIFA Men's World Ranking points for the CAF zone, ranking dated 20 July 2026,\footnote{FIFA/CAF World Ranking, \url{https://www.fifa.com/fifa-world-ranking}, ranking date 20 July 2026.} used as a continuously scaled, currently available proxy for footballing performance $Y_i$ in the absence of a bespoke AFCON points panel. The full dataset (GDP, population, FIFA points, estimated residuals, and code reproducing every table and figure in this paper) is provided as a reproducible supplementary file accompanying this submission.

\subsubsection{Cross-sectional estimation and statistical inference} \label{sec:inference}
We estimate, by OLS on this single cross-section of $N=24$ countries,
\[
\ln(\mathrm{FIFA}_i) = \ln A + \alpha_1 \ln GDP_i + \alpha_2 \ln Pop_i + z_i,
\]
using the corrected, non-collinear specification of \Cref{sec:methodology}. \Cref{tab:regression} reports the estimated coefficients together with heteroskedasticity-robust (HC3) standard errors, appropriate given the small sample ($N=24$).

\begin{table}[h]
\centering
\caption{Cross-sectional OLS estimates, $\ln(\mathrm{FIFA}_i)$ on $\ln GDP_i$ and $\ln Pop_i$ ($N=24$). HC3 robust standard errors in parentheses.}
\label{tab:regression}
\begin{tabular}{lrrrr}
\toprule
 & Coefficient & Robust s.e. & $t$ & $p$-value \\
\midrule
Constant ($\ln A$) & 6.9526 & 0.0898 & 77.41 & $<0.001$ \\
$\ln GDP_i$ ($\alpha_1$) & 0.0606 & 0.0336 & 1.80 & 0.072 \\
$\ln Pop_i$ ($\alpha_2$) & 0.0063 & 0.0293 & 0.21 & 0.831 \\
\midrule
\multicolumn{5}{l}{$R^2 = 0.347$, adjusted $R^2 = 0.285$} \\
\multicolumn{5}{l}{Joint test $H_0: \alpha_1=\alpha_2=0$: $F(2,21)=5.17$, $p=0.015$} \\
\bottomrule
\end{tabular}
\end{table}

Three findings temper a purely mechanical reading of the elasticities. First, the individual GDP elasticity $\hat\alpha_1=0.061$ is only marginally significant at conventional levels ($p=0.072$ with HC3 standard errors), and the population elasticity $\hat\alpha_2=0.006$ is statistically indistinguishable from zero ($p=0.831$); taken individually, neither coefficient offers strong evidence of a GDP or population effect on FIFA points. Second, the \emph{joint} test $H_0:\alpha_1=\alpha_2=0$ is nonetheless rejected at the 5\% level ($F=5.17$, $p=0.015$): GDP and population together explain a statistically significant, if modest, share of cross-country variation in footballing performance. Third, the fit is modest in an absolute sense ($R^2=0.35$): GDP and population together explain about a third of the cross-country variation in FIFA points, leaving the bulk to the residual $z_i = \ln(\mathrm{FIFA}_i) - \widehat{\ln f}(GDP_i,Pop_i)$. Rather than treating this residual as direct evidence of an efficiency channel, \Cref{sec:circularity} below quantifies exactly how much of that residual variation is, by construction, simply a repackaging of the raw FIFA ranking.

We take $\bar z_i = z_i$ as each country's current structural log-efficiency level and $f_i=\exp(\widehat{\ln f}(GDP_i,Pop_i))$ as its structural GDP/population-based potential, and simulate \Cref{eq:sde} forward from $z_{i,0}=\bar z_i$ with common structural parameters $\gamma=1$, $c=1$, and identical constant investment $u=0.2$ for every country over $T=5$ years. We flag explicitly that setting $\bar z_i=z_i$ -- i.e., treating each country's currently estimated residual as its long-run structural target -- is an \emph{identification assumption}, not an innocuous normalization: it is what allows us to simulate forward from a single cross-section without a panel estimate of $\bar z_i$, and it mechanically preserves each country's current relative standing rather than allowing the model to independently discover it. A genuine panel estimate of $\bar z_i$ (Step 2 of \Cref{sec:estimation}) is needed to relax this assumption; see \Cref{sec:conclusion}. \\

\textbf{An attempted calibration of $\kappa,\sigma$, and why it was not used.} We attempted to calibrate the shared dynamic parameters $\kappa$ (mean-reversion speed) and $\sigma$ (volatility) to real data, using verified AFCON results (podium finishes and quarter-finalists, cross-checked against official records) for the last four 24-team editions (2019, 2021, 2023, 2025). Identifying $\kappa,\sigma$ from an AR(1)/Ornstein--Uhlenbeck regression requires country-level observations in \emph{two or more consecutive} editions, but the great majority of the 24 countries have an independently verified result (top-4 finish or quarter-final exit) in only one of the four editions -- the remaining cells would have to be filled in from unverified recollection rather than sourced results. With only a handful of usable consecutive pairs, a regression-based $\hat\kappa,\hat\sigma$ would be dominated by sampling noise rather than signal, so we do not report one: a spurious point estimate presented as ``calibrated'' would be worse than being explicit about the limitation. Genuinely identifying $\kappa_i,\sigma_i$ requires the kind of multi-cycle panel (AFCON points \emph{and} time-varying GDP, several editions deep, for every country) that the Kalman-filter strategy of Step 2, \Cref{sec:estimation}, is built for; assembling that panel is future work. Here we instead run $\kappa,\sigma$ at an illustrative baseline ($\kappa=0.7,\sigma=0.30$) and directly check, below, how much the results depend on that choice.

\subsubsection{Quantifying circularity in the empirical exercise}
\label{sec:circularity}

With $\hat\alpha_1=0.061$ and $\hat\alpha_2=0.006$, the fitted production function is close to flat in the two economic regressors, so the residual $z_i$ is close to a re-centered version of $\ln(\mathrm{FIFA}_i)$. We quantify this explicitly rather than leave it implicit. By construction of the OLS residual, $\mathrm{Corr}(z_i,\ln(\mathrm{FIFA}_i)) = \sqrt{1-R^2}$; with $R^2=0.347$ this gives
\[
\mathrm{Corr}\big(z_i,\, \ln(\mathrm{FIFA}_i)\big) = \sqrt{1-0.347} = 0.808,
\]
which we also confirm numerically from the estimated residuals ($0.8082$). In other words, roughly $1-R^2\approx 65\%$ of the cross-country variance in log FIFA points is, mechanically, still present in the ``efficiency'' residual $z_i$. This is not a substantive empirical finding about football; it is an algebraic identity that holds for \emph{any} regression with this $R^2$, and it means that a ranking of countries by $z_i$ will necessarily look similar to a ranking by raw FIFA points. Consistent with this, the Spearman rank correlation between the country ranking implied by $z_i$ and the raw FIFA-points ranking is $0.734$ ($p<0.001$).

We draw two conclusions from this. First, statements of the form ``Senegal and Morocco outperform their GDP-implied potential'' should be read, and are reported throughout this paper, as \emph{conditional on the weak explanatory power of the fitted production function}: given how little of FIFA points GDP and population jointly explain ($R^2=0.35$), a claim that country $i$ has high $z_i$ is close to, though not identical to, the more mundane claim that country $i$ has a high FIFA ranking. Second, and more constructively, this is precisely why we relabel $z_i$ throughout this paper as a \emph{conditional} performance residual (performance conditional on GDP and population) rather than as a validated structural efficiency parameter (\Cref{sec:sfa-dea}): the residual almost certainly also absorbs omitted structural factors that are correlated with football performance but not modeled here -- diaspora and expatriate players, football tradition, institutional quality \citep{batarfireade2021}, and working-age male population (15--35) rather than total population. Because we do not have reliable cross-country data on these factors for the present sample, we do not add them as additional controls in this version of the paper; instead we are explicit about the limitation here and in \Cref{sec:conclusion}, and we requalify $z_i$ accordingly everywhere it is used, including in the abstract.

\subsection{Results}
\Cref{fig:residuals} ranks the 24 teams by $z_i$. Senegal ($z=0.221$) and Morocco ($z=0.207$) are the largest positive residuals -- both current AFCON champions/finalists, over-performing their GDP and population by a wide margin -- followed by Mali, C\^ote d'Ivoire, Comoros, and Cameroon. Zimbabwe ($z=-0.190$), Tanzania, Botswana, Sudan, and Angola are the largest negative residuals, with South Africa, despite having the single largest GDP in the sample (\$427bn), at $z=-0.065$ and ranking only 10th out of 24 in simulated expected performance.

\begin{figure}[h!]
\centering
\includegraphics[width=0.6\linewidth]{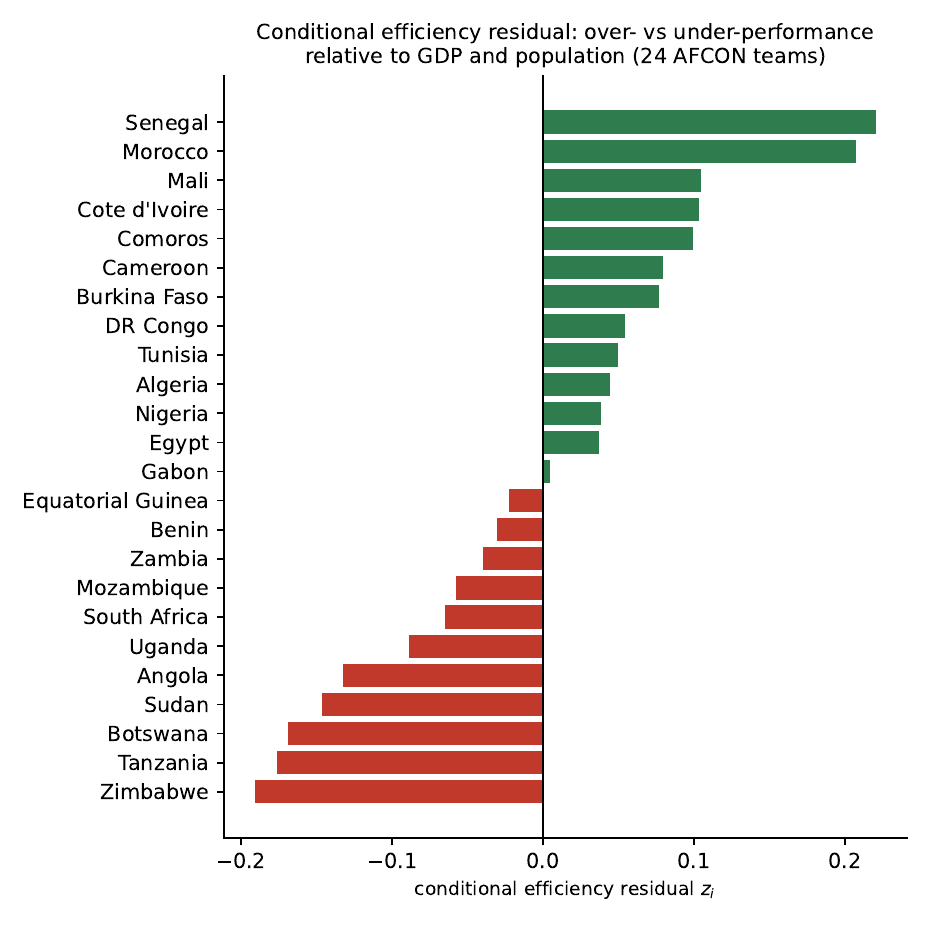}
\caption{Conditional efficiency residual $z_i$ for the 24 AFCON 2025 teams, from the cross-sectional regression of $\ln(\mathrm{FIFA}_i)$ on $\ln GDP_i$ and $\ln Pop_i$.}
\label{fig:residuals}
\end{figure}

\Cref{fig:scatter} plots $\ln GDP$ against $\ln(\mathrm{FIFA})$ directly, colored by $z_i$: countries above the fitted line (green) systematically translate economic capacity into results more effectively than those below it (red). \Cref{tab:table2} reports the full underlying data.

\begin{figure}[h!]
\centering
\includegraphics[width=0.65\linewidth]{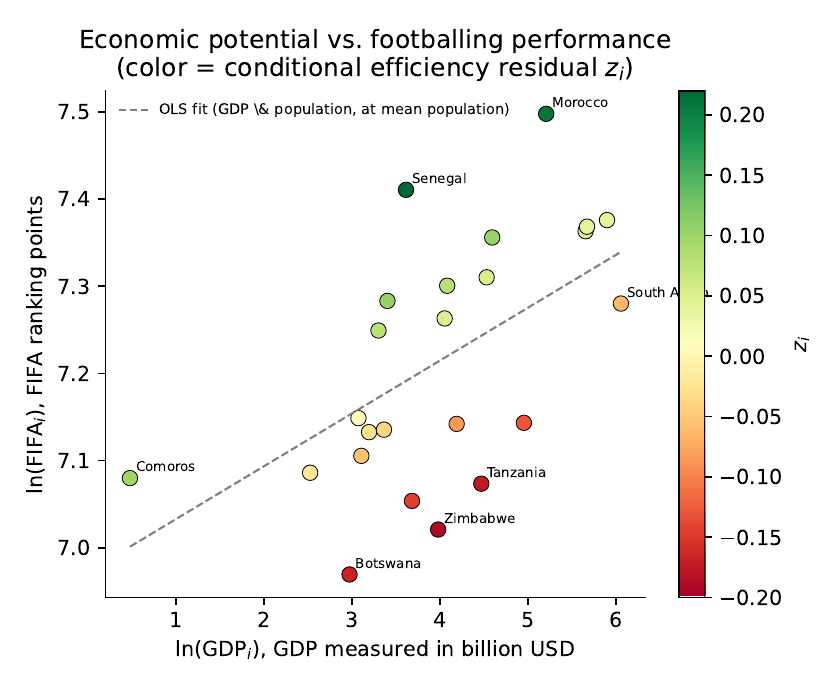}
\caption{GDP versus FIFA ranking points (logs) for the 24 AFCON 2025 teams, colored by the conditional efficiency residual $z_i$; dashed line is the OLS fit (evaluated at mean population).}
\label{fig:scatter}
\end{figure}

\Cref{fig:rankscatter} compares each country's GDP rank to its simulated-performance rank (baseline $\kappa=0.7,\sigma=0.30$, Monte Carlo, $N=8{,}000$ paths); the rank correlation is $0.62$ -- positive, as expected, but far from 1, and several countries reverse rank by ten places or more: Senegal moves from GDP rank 15 to performance rank 2, and Tanzania moves from GDP rank 9 to performance rank 21. This is, to our knowledge, the first time this specific theoretical mechanism has been confronted -- even in this simple, single-cross-section, illustrative way -- with real GDP, population, and footballing-strength data for the full AFCON 2025 field.

\begin{figure}[h!]
\centering
\includegraphics[width=0.65\linewidth]{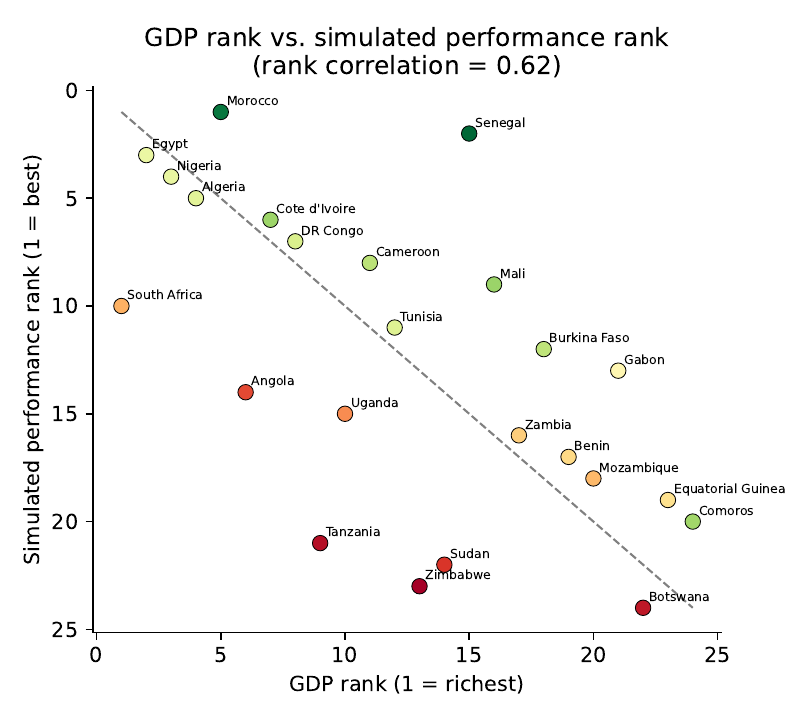}
\caption{GDP rank versus simulated expected-performance rank for the 24 AFCON 2025 teams (rank correlation $=0.62$); points below the diagonal over-perform their GDP rank.}
\label{fig:rankscatter}
\end{figure}

\begin{table}[h!]
\centering
\caption{Ranking: GDP vs.\ simulated performance ($T=5$, common $u=0.2$, baseline $\kappa=0.7,\sigma=0.30$).}
\label{tab:table2}
\small
\begin{tabular}{lrrrrrrr}
\toprule
Country & GDP (\$bn) & Pop (M) & FIFA & $z_i$ & GDP rk & Perf.\ rk & $\Delta$rk \\
\midrule
South Africa & 427.1 & 63.1 & 1451 & $-0.065$ & 1 & 10 & $-9$ \\
Egypt & 364.6 & 107.9 & 1597 & 0.037 & 2 & 3 & $-1$ \\
Nigeria & 290.5 & 237.5 & 1585 & 0.038 & 3 & 4 & $-1$ \\
Algeria & 285.7 & 47.3 & 1577 & 0.044 & 4 & 5 & $-1$ \\
Morocco & 182.6 & 37.7 & 1804 & 0.207 & 5 & 1 & $+4$ \\
Angola & 141.7 & 39.2 & 1266 & $-0.132$ & 6 & 14 & $-8$ \\
C\^ote d'Ivoire & 98.9 & 32.9 & 1566 & 0.103 & 7 & 6 & $+1$ \\
DR Congo & 92.8 & 106.6 & 1496 & 0.054 & 8 & 7 & $+1$ \\
Tanzania & 87.3 & 67.6 & 1180 & $-0.176$ & 9 & 21 & $-12$ \\
Uganda & 66.0 & 48.0 & 1264 & $-0.088$ & 10 & 15 & $-5$ \\
Cameroon & 59.3 & 29.9 & 1481 & 0.080 & 11 & 8 & $+3$ \\
Tunisia & 57.6 & 12.4 & 1427 & 0.049 & 12 & 11 & $+1$ \\
Zimbabwe & 53.5 & 17.4 & 1120 & $-0.190$ & 13 & 23 & $-10$ \\
Sudan & 39.7 & 50.4 & 1157 & $-0.146$ & 14 & 22 & $-8$ \\
Senegal & 37.1 & 19.2 & 1653 & 0.221 & 15 & 2 & $+13$ \\
Mali & 30.1 & 25.3 & 1456 & 0.104 & 16 & 9 & $+7$ \\
Zambia & 28.9 & 21.9 & 1256 & $-0.040$ & 17 & 16 & $+1$ \\
Burkina Faso & 27.1 & 24.1 & 1407 & 0.077 & 18 & 12 & $+6$ \\
Benin & 24.4 & 14.9 & 1252 & $-0.030$ & 19 & 17 & $+2$ \\
Mozambique & 22.3 & 35.9 & 1219 & $-0.058$ & 20 & 18 & $+2$ \\
Gabon & 21.6 & 2.3 & 1273 & 0.005 & 21 & 13 & $+8$ \\
Botswana & 19.5 & 2.6 & 1064 & $-0.169$ & 22 & 24 & $-2$ \\
Equatorial Guinea & 12.5 & 1.6 & 1195 & $-0.022$ & 23 & 19 & $+4$ \\
Comoros & 1.6 & 0.9 & 1188 & 0.099 & 24 & 20 & $+4$ \\
\bottomrule
\end{tabular}
\end{table}

\textbf{Sensitivity to $\kappa,\sigma$.} Since $\kappa,\sigma$ are not calibrated, we check how much the headline results depend on them, over a grid $\kappa\in\{0.3,0.5,0.7,1.0,1.5\}$, $\sigma\in\{0.10,0.20,0.30,0.40,0.50\}$ (25 combinations spanning slow/fast mean reversion and low/high volatility). \Cref{fig:heatmap} plots the resulting GDP-rank/performance-rank correlation: it stays in a narrow band, $[0.605,0.643]$, across the entire grid (baseline: $0.621$). The three worst performers relative to GDP -- Botswana, Sudan, and Zimbabwe -- are identical across every one of the 25 $(\kappa,\sigma)$ combinations tested; the three best performers are Egypt, Morocco, and Senegal at all but a small number of grid points, where Monte Carlo sampling noise occasionally reorders two closely-ranked countries. This is expected given the model's structure: every country is simulated starting from its own estimated steady state $z_{i,0}=\bar z_i$ under the same $\kappa,\sigma,\gamma,c,u$, so $\kappa,\sigma$ shift every country's simulated path in a similar way and largely wash out of the \emph{relative} ranking, which is driven overwhelmingly by the cross-sectional $z_i$ estimated in \Cref{sec:inference}'s regression -- exactly the quantity that \emph{is} identified from current data. This is a reassuring, if partial, robustness result: it does not substitute for a genuine panel calibration of $\kappa,\sigma$, but it shows the paper's central ranking claim is not an artifact of the particular illustrative values chosen for the un-identified dynamic parameters.

\begin{figure}[h!]
\centering
\includegraphics[width=0.55\linewidth]{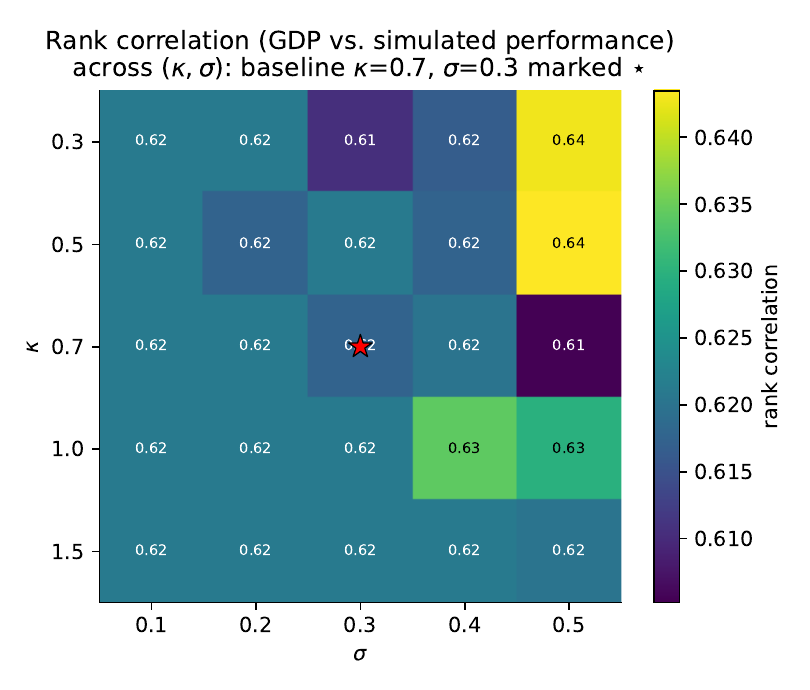}
\caption{Rank correlation between GDP and simulated performance across a $5\times5$ grid of $(\kappa,\sigma)$ values; the baseline $\kappa=0.7,\sigma=0.30$ is marked with a star.}
\label{fig:heatmap}
\end{figure}

\textbf{Remarks.} This exercise uses FIFA ranking points, not AFCON tournament points, as the performance measure $Y_i$, and a single cross-section rather than a panel, so $z_i$ should be read as ``footballing strength relative to GDP and population, as of mid-2026,'' and, per \Cref{sec:circularity}, as a residual that mechanically correlates with raw FIFA points at $\rho=0.81$; and $\gamma,c$, and (per the discussion above) $\kappa,\sigma$, are shared across countries and set at illustrative values rather than estimated individually. It nonetheless directly operationalizes, for the first time in this paper, the mechanism the theoretical model was built to capture, and the sensitivity analysis shows its central ranking result is not an artifact of the un-identified parameters.

\subsection{Discussion}

The model and its numerical illustration yield several insights:
\begin{itemize}
\item Economic resources determine potential, not performance: $f(X_{i,t})$ sets the scale, but $\theta_{i,t}$ determines how much of that potential is realized.
\item Efficiency is the key driver of success, and, because performance is multiplicative in efficiency, gains in efficiency compound: the optimal policy directs more investment toward already-efficient federations (\Cref{sec:numerics}), which helps explain the persistence of over- and under-performance relative to GDP across cycles.
\item Stochastic shocks explain variability in outcomes; volatility $\sigma_i$ widens the distribution of possible results without shifting the long-run target $\bar z_i$.
\item Optimal investment policies depend on structural parameters ($\kappa_i,\gamma_i,c_i$) and on the remaining time horizon, and are only well defined once a realistic investment budget cap is imposed.
\item The conditional efficiency residual used in \Cref{sec:results} is informative but should not be over-read: roughly two-thirds of its cross-country variance is, by construction, a restatement of the raw performance ranking (\Cref{sec:circularity}), and the remainder still mixes any genuine efficiency effect with omitted structural factors (diaspora, institutional quality, demographic structure) that are not separately controlled for here.
\end{itemize}

In particular, the cross-country illustration of \Cref{sec:numerics} shows, in a controlled numerical setting, that a country with moderate GDP but high efficiency can outperform a richer, less efficient country, and the 24-country cross-section of \Cref{sec:results} shows a broadly consistent pattern in real GDP, population, and FIFA ranking data for the actual AFCON 2025 field -- Senegal and Morocco substantially over-perform their economic potential, while South Africa, despite the largest GDP in the sample, and Tanzania and Zimbabwe under-perform theirs. This is qualitatively consistent with the empirical evidence in \citet{luizfadal2011,batarfireade2021}, though, as emphasized throughout, the strength of that consistency is bounded by the weak explanatory power ($R^2=0.35$) of the underlying production-function regression.

The model suggests that:
\begin{itemize}
\item Improving governance and training systems is more effective than increasing budgets alone -- as efficiency rises, increasingly, early institutional gains are worth protecting and building on.
\item Strategic investment in football infrastructure yields long-term efficiency gains, but is subject to a genuine budget constraint: unlimited investment is not just unrealistic but, absent a cap, is not even well defined as an optimization problem (\Cref{sec:control}).
\item Stability (low $\sigma_i$) reduces volatility and improves the reliability, though not necessarily the mean level, of expected performance.
\end{itemize}

\section{Conclusion and future work}\label{sec:conclusion}
This paper proposes a stochastic optimization framework for understanding football performance in Africa. By modeling efficiency as a controlled stochastic process, deriving and verifying the associated Hamilton--Jacobi--Bellman equation under a discounted objective, and positioning the resulting efficiency measure explicitly relative to the stochastic frontier and DEA traditions, we provide a theoretical explanation for the observed disconnect between economic capacity and sporting success, together with a normative optimal investment rule. We show that a bounded investment budget is required for well-posedness, give a verification theorem under explicit growth conditions, validate the numerical solution by a grid-refinement convergence study, and illustrate the resulting optimal policy and its cross-country implications by simulation. Applied to real data for the 24 AFCON 2025 national teams, the cross-sectional exercise reproduces the paper's central qualitative claim -- Senegal and Morocco over-perform their GDP-implied potential, while South Africa, Tanzania, and Zimbabwe under-perform theirs -- while being explicit that roughly two-thirds of the underlying residual's cross-country variance is, by construction, a restatement of the raw FIFA ranking, and that the residual still absorbs factors not modeled here.

Future work may extend this framework to:
\begin{itemize}
\item a genuine multi-cycle panel estimate of $(\kappa_i,\sigma_i,\bar z_i)$ (AFCON tournament points, and time-varying GDP, several editions deep, for every country) via the Kalman-filter strategy of \Cref{sec:estimation}, which alone can identify country-specific dynamic parameters and relax the $\bar z_i=z_i$ identification assumption used in \Cref{sec:results};
\item additional controls for the factors most likely to be absorbed into the current residual -- diaspora and expatriate-player data, an institutional-quality index, and working-age (15--35) male population in place of total population -- together with a properly specified stochastic frontier (SFA) decomposition of noise versus one-sided inefficiency, to reduce reliance on the symmetric-residual shortcut used here (\Cref{sec:sfa-dea,sec:circularity});
\item player-level data and network effects (diaspora) and dynamic tournament modeling with strategic interaction between federations;
\item analytical approximations to the HJB equation (e.g., local quadratic expansions in the spirit of \citet{dixit1991}) to complement the numerical solution.
\end{itemize}

\bibliographystyle{plainnat}   
\bibliography{references}       

\end{document}